\documentclass[12pt,reqno]{article}
\usepackage{mathrsfs}
\usepackage{amsfonts}
\usepackage{cite}

\usepackage{amsmath,amsthm,amsfonts,amssymb,latexsym,mathrsfs,color}

\newtheorem{thm}{Theorem}[section]
\newtheorem{lem}{Lemma}[section]
\newtheorem{prop}{Proposition}[section]
\newtheorem{conj}{Conjecture}[section]
\newtheorem{prob}{Problem}[section]
\newtheorem{coro}{Corollary}[section]

\theoremstyle{definition}
\newtheorem{exm}{Example}[section]
\newtheorem{rem}{Remark}[section]

\numberwithin{equation}{section}

\title{Operations of graphs and unimodality of independence polynomials
\thanks{Supported partially by the National Science Foundation of China (Nos.11071030,11201191), Natural Science Foundation of Jiangsu Higher
Education Institutions (No.12KJB110005) and PAPD of Jiangsu Higher
Education Institutions.
\newline\hspace*{5mm}
   {\it Email addresses:}
    bxzhu@jsnu.edu.cn (B.-X. Zhu)}}
\author{Bao-Xuan Zhu}
\date{\footnotesize School of Mathematical Sciences,
          Jiangsu Normal University,
         Xuzhou 221116, P. R. China}

\begin{document}
\maketitle
\begin{abstract}
Given two graphs $G$ and $H$, assume that
$\mathscr{C}=\{C_1,C_2,\ldots, C_q\}$ is a clique cover of $G$ and
$U$ is a subset of $V(H)$. We introduce a new graph operation called
the clique cover product, denoted by $G^{\mathscr{C}}\star H^U$, as
follows: for each clique $C_i\in \mathscr{C}$, add a copy of the
graph $H$ and join every vertex of $C_i$ to every vertex of $U$. We
prove that the independence polynomial of $G^{\mathscr{C}}\star H^U$
$$I(G^{\mathscr{C}}\star
H^U;x)=I^q(H;x)I(G;\frac{xI(H-U;x)}{I(H;x)}),$$ which generalizes
some known results on independence polynomials of corona and rooted
products of graphs obtained by Gutman and Rosenfeld, respectively.
Based on this formula, we show that the clique cover product of some
special graphs preserves symmetry, unimodality, log-concavity or
reality of zeros of independence polynomials. As applications we
derive several known facts in a unified manner and solve some
unimodality conjectures and problems.
\\[7pt]
\end{abstract}

\section{Introduction}
\hspace*{\parindent}
For the graph theoretical terms used but not defined, we follow
Bondy and Murty~\cite{BM76}.

Let $G=(V(G),E(G))$ be a finite and simple graph. By $G-U$ we mean
the induced subgraph $G[V-U]$, if $U\subseteq V(G)$. We also denote
by $G-e$ the subgraph of $G$, obtained by deleting an edge $e$ of
$E(G)$. For $v\in V(G)$, denote $N(v)=\{w: \text{$w\in V(G)$ and
$vw\in E(G)$}\}$ and $N[v]=N(v)\cup \{v\}$. The {\it join} of two
disjoint graphs $G_1$ and $G_2$, denoted by $G_1+G_2$,  with
$E(G_1)\cup E(G_2)\cup \left\{uv:u\in V(G_1),v\in V(G_2)\right\}$ as
the edge set and $V(G_1)\cup V(G_2)$ as the vertex set. An {\it
independent set} in a graph $G$ is a set of pairwise non-adjacent
vertices. A {\it maximum independent set} in $G$ is a largest
independent set and its size is denoted $\alpha(G)$. Let $i_k(G)$
denote the number of independent sets of cardinality $k$ in $G$.
Then its generating function
$$I(G;x)=\sum\limits_{k=0}^{\alpha(G)}i_k(G)x^k,\quad i_0(G)=1$$
is called the {\it independence polynomial} of $G$ (Gutman and
Harary~\cite{GH83}). In general, it is an NP-complete problem to
determine the independence polynomial, since evaluating $\alpha(G)$
is an NP-complete problem~(\cite{GJ79}). Thus, a classical question
is to compute the independence polynomial of a graph. It is a good
way to obtain the independence polynomial of a graph in term of its
subgraphs. One can easily deduce (e.g., Gutman and
Harary~\cite{GH83}) that
\begin{eqnarray*}
I(G_1\cup G_2;x)=I(G_1;x)I(G_2;x),\quad I(G_1+
G_2;x)=I(G_1;x)+I(G_2;x)-1.
\end{eqnarray*}

As we know, to study properties of graphs, there are some useful and
important operations of graphs in the graph theory. Motivated by the
above mentioned examples, one may further ask which operation of
graphs is good to compute the independence polynomial. Recall that a
{\it clique cover} of a graph $G$ is a spanning graph of $G$, each
component of which is a clique. We here define an operation of
graphs called the {\it clique cover product}. Given two graphs $G$
and $H$, assume that $\mathscr{C}=\{C_1,C_2,\ldots, C_q\}$ is a
clique cover of $G$ and $U$ is a subset of $V(H)$. Construct a new
graph from $G$, as follows: for each clique $C_i\in \mathscr{C}$,
add a copy of the graph $H$ and join every vertex of $C_i$ to every
vertex of $U$. Let $G^{\mathscr{C}}\star H^U$ denote the new graph.
In fact, the clique cover product of graphs is a common
generalization of some known operations of graphs. For instance: If
each clique $C_i$ of the clique cover $\mathscr{C}$ is a vertex,
then $G^{V(G)}\star H^{V(H)}$ is the corona of $G$ and $H$ defined
by Frucht and Harary~\cite{FH70}, denoted by $G\circ H$. If each
clique $C_i$ of the clique cover $\mathscr{C}$ is a vertex and a
vertex $v$ is the root of $H$, then $G^{V(G)}\star (H-v)^{N(v)}$ is
the rooted product of $G$ and $H$ introduced by Godsil and
MacKay~\cite{GM78}, denoted by $G\overline{\circ} H$. If we take
$H=2K_1$ and $U=V(2K_1)$, then $G^{\mathscr{C}}\star H^U$ is the
graph $\mathscr{C}(G)$ obtained by Stevanovi\'c~\cite{Ste98} using
the clique cover construction. Note that $G^{\mathscr{C}}\star
H^{V(H)}$ also contains the compound graph $G^{\Delta}(H)$ as the
special case \cite{SSW11}. As the basic result of this paper, we
formulate the independence polynomial for the clique cover product
as follows.
\begin{thm}\label{thm-CCP}
Given two graphs $G$ and $H$, assume that $\mathscr{C}$ is a clique
cover of $G$ and $U$ is a subset of $V(H)$. Let $|\mathscr{C}|=q$.
Then we have
 $$I(G^{\mathscr{C}}\star H^U;x)=I^q(H;x)I\left(G;\frac{xI(H-U;x)}{I(H;x)}\right).$$
\end{thm}
In view of Theorem~\ref{thm-CCP}, the following is immediate.
\begin{coro}
Given two graphs $G$ and $H$, assume that $\mathscr{C}$ is a clique
cover of $G$ and $U$ is a subset of $V(H)$. Let $|\mathscr{C}|=q$.
Then $I^{q-\alpha(G)}(H;x)$ divides  $I(G^{\mathscr{C}}\star
H^U;x)$.
\end{coro}
In fact, Theorem~\ref{thm-CCP} has generalized some known results,
e.g., the following corollaries obtained by Gutman~\cite{Gut92} and
Rosenfeld~\cite{R09}, respectively. In addition, our method is
different from theirs.
\begin{coro}\emph{\cite[Gutman]{Gut92}}\label{coro-Gut}
For any graph $G$ of order $n$, $I(G\circ
H;x)=I^n(H;x)I\left(G;\frac{x}{I(H;x)}\right)$.
\end{coro}
\begin{coro}\emph{\cite{Gut92,R09}}\label{coro-root product}
If $G$ is a graph of order $n$ and $H$ is a graph with root $v$,
then $$I(G \overline{\circ}
H;x)=I^n(H-v;x)I\left(G;\frac{xI(H-N[v];x)}{I(H-v;x)}\right).$$
\end{coro}
\begin{coro}\emph{\cite{R09}}
Let $H$ be a graph with root $v$, where $v$ is a pedant vertex and
$N(v)=u$. If $G$ is a graph of order $n$,  then
$$I(G \overline{\circ}
H;x)=I^n(H-v;x)I\left(G;\frac{xI(H-v-u;x)}{I(H-v;x)}\right).$$
\end{coro}

Let $a_0,a_1,\ldots,a_n$ be a sequence of nonnegative numbers. It is
{\it unimodal} if there is some $m$, called a {\it mode} of the
sequence, such that
$$a_0\le a_1\le\cdots\le a_{m-1}\le a_m\ge a_{m+1}\ge\cdots\ge a_n.$$
It is {\it log-concave} if $a_k^2\ge a_{k-1}a_{k+1}$ for all $1\le k
\le n-1$.  It is {\it symmetric} if $a_k=a_{n-k}$ for $0\le k\le n$.
Clearly, a log-concave sequence of positive numbers is unimodal
(see, e.g., Brenti~\cite{Bre89}). We say that a polynomial
$\sum_{k=0}^na_kx^k$ is {\it unimodal} ({\it log-concave}, {\it
symmetric}, respectively) if the sequence of its coefficients
$a_0,a_1,\ldots,a_n$ is unimodal (log-concave, symmetric,
respectively). A mode of the sequence $a_0,a_1,\ldots,a_n$ is also
called a mode of the polynomial $\sum_{k=0}^na_kx^k$. Unimodality
problems arise naturally in many branches of mathematics and have
been extensively investigated. See Stanley's survey~\cite{Sta89} and
Brenti's supplement~\cite{Bre94} for known results and open problems
on log-concave and unimodal sequences arising in algebra,
combinatorics and geometry. It is well known that if the polynomial
$\sum_{k=0}^{n}a_kx^k$ has only real zeros for the nonnegative
sequence $\{a_i\}$, then by Newton's inequalities the sequence
$\{a_n\}$ is log-concave and unimodal (see Hardy, Littlewood and
P\'olya~\cite[p. 104]{HLP52}).

Recently, the open problems on the unimodality~(Read~\cite[p.
68]{Rea68}) and log-concavity ~(Welsh~\cite[p. 266]{Wel76}) of the
chromatical polynomial of a graph  has been solved by
Huh~\cite{h12}. On the other hand, unimodality problems and zeros of
independence polynomials have been investigated, e.g., see
\cite{AMSE87,B10,BS10,BHN04,BN05,CW10,CS07,HL94,Ham90,LM02,LM04WSEAS,LM04CJM,LM05,LM06EJC,LM06CN,LM07,LM08,Man09,WZ10,Zhu07}
for an extensive literature in recent years. It is well known that
the matching polynomial of a graph has only real zeros~\cite{HL72,
Sch81}. In fact, the independence polynomial can be regarded as a
generalization of the matching polynomial because the matching
polynomial of a graph and the independence polynomial of its line
graph are identical. Wilf asked whether the independence polynomials
are also unimodal. However Alavi, Malde, Schwenk,
Erd\H{o}s~\cite{AMSE87} gave a negative example. They further
conjectured the following.
\begin{conj} The independence polynomial of any tree or forest is unimodal.
\end{conj}
Similarly, Levit and Mandrescu~\cite{LM03} also make the next
conjecture.

\begin{conj}
The independence polynomial of any very well-covered graph is
log-concave.
\end{conj}

Whereas the independence polynomials for certain special classes of
graphs are unimodal and even have only real zeros. For instance, the
independence polynomial of a line graph has only real zeros. More
generally, the independence polynomial of a claw-free graph has only
real zeros (Chudnovsky and Seymour~\cite{CS07}). Although the
independence polynomial of almost every graph of order $n$ has a
nonreal zero, the average independence polynomials always have all
real and simple zeros (Brown and Nowakowski~\cite{BN05}). Hence an
interesting problem naturally arises.

\begin{prob}\emph{\cite{Bro08}}
When does the independence polynomial of a graph have only real
zeros ?
\end{prob}

The symmetry of the matching polynomial and the characteristic
polynomial of a graph were observed (see \cite{GM78,Ken92} for
instance). Thus, we naturally study the symmetric independence
polynomial. Stevanovi\'c~\cite{Ste98} showed the next result.
\begin{thm}\emph{\cite{Ste98}}\label{ste}
If there is an independent set $S$ in $G$ such that $|N(A)\cap
S|=2|A|$ holds for every independent set $A \subseteq V (G)-S$, then
$I(G; x)$ is symmetric.
\end{thm}
By virtue of this result, Stevanovi\'c deduced the following
corollary.

\begin{coro}
\begin{itemize}
\item [\rm (i)]
If a graph $G$ has $i_{\alpha(G)}(G)=1$, $i_{\alpha(G)-1}(G)=|V(G)|$
and the unique maximum independence set $S$ satisfies $|N(u)\cap
S|=2$ for every vertex $u\in V-S$, then $I(G;x)$ is symmetric.
\item [\rm (ii)]
If $G$ is a claw-free graph with $i_{\alpha(G)}=1$, $i_{\alpha(G)-1}
= |V (G)|$, then $I(G; x)$ is symmetric.
\end{itemize}
\end{coro}
According to this corollary, Stevanovi\'c~\cite{Ste98} further
obtained a few ways to construct graphs having symmetric
independence polynomials, e.g., $G\circ 2K_1$ and $\mathscr{C}(G)$.
However, the following general problem is still open.

\begin{prob}
When is the independence polynomial of a graph symmetric ?
\end{prob}

The organization of this paper is as follows. In Section $2$, we
give the proof of Theorem~\ref{thm-CCP}. In Section $3$, based on
the formula in Theorem~\ref{thm-CCP}, we present various results
that the clique cover product of some special graphs preserves
symmetry, unimodality, log-concavity or reality of zeros of
independence polynomials. As applications we derive several known
facts and solve some unimodality conjectures and problems in a
unified manner. Finally, in the concluding remarks, we also give the
similar result for another new graphs operation called the cycle
cover product.

\section{The proof of Theorem~\ref{thm-CCP}}
\hspace*{\parindent}
\begin{proof} We here give a combinatorial proof. Let
$I(G;x)=\sum_{k=0}^{\alpha(G)}s_k x^k$,
$I(H;x)=\sum_{k=0}^{\alpha(H)}a_k x^k$,
$I(H-U;x)=\sum_{k=0}^{\alpha(H-U)}b_k x^k$  and
$I(G^{\mathscr{C}}\star H^U;x)=\sum_{k\geq0} t_k x^k$. We can give
an explicit expression of $t_k$ in the following method. For each
$k$, select $k$ independent elements from $V(G^{\mathscr{C}}\star
H^U)$ in a two-stage process. First, let us choose $m$ independent
elements from the $V(G)$. And then select the remaining $(k-m)$
independent elements from $V(m(H-U)\bigcup (q-m)H)$. The number of
ways in which we make two choices is $s_m$ and
$$\sum\limits_{i_1+i_2+\ldots+i_m+j_1+j_2+\ldots+j_{q-m}=k-m}b_{i_1}b_{i_2}\cdots
b_{i_{m}}a_{j_1}a_{j_2}\cdots a_{j_{q-m}}$$ respectively. In
consequence, we obtain that
$$t_k=\sum\limits_{m=0}^k s_m\sum\limits_{i_1+i_2+\ldots+i_m+j_1+j_2+\ldots+j_{q-m}=k-m}b_{i_1}b_{i_2}\cdots
b_{i_{m}}a_{j_1}a_{j_2}\cdots a_{j_{q-m}}.$$ Thus we have
\begin{eqnarray*}
I(G^{\mathscr{C}}\star H^U;x)&=&\sum\limits_{k\geq0} t_k x^k\\
&=&\sum\limits_{k\geq0}\left(\sum\limits_{m=0}^k
s_m\sum\limits_{i_1+i_2+\ldots+i_m+j_1+j_2+\ldots+j_{q-m}=k-m}b_{i_1}b_{i_2}\cdots
b_{i_{m}}a_{j_1}a_{j_2}\cdots a_{j_{q-m}}\right)x^k\\
&=&\sum\limits_{m=0}^{\alpha(G)} s_m x^m\sum\limits_{k\geq m}
\left(\sum\limits_{i_1+i_2+\ldots+i_m+j_1+j_2+\ldots+j_{q-m}=k-m}b_{i_1}b_{i_2}\cdots
b_{i_{m}}a_{j_1}a_{j_2}\cdots a_{j_{q-m}}x^{k-m}\right)\\
&=&\sum\limits_{m=0}^{\alpha(G)} s_m I^{m}(H-U;x) I^{q-m}(H;x)x^m\\
&=&I^q(H;x)I\left(G;\frac{x I(H-U;x)}{I(H;x)}\right),
\end{eqnarray*}
which is desired.
\end{proof}
\section{Various results
for symmetry, unimodality, log-concavity or reality of zeros of
independence polynomials.} \hspace*{\parindent}
In this section, based on the formula in Theorem~\ref{thm-CCP}, we
present various results that the clique cover product of some
special graphs preserves symmetry, unimodality, log-concavity or
reality of zeros of independence polynomials. We not only derive
several known facts in a unified manner, but also settle some
unimodality conjectures and problems in the literature.

Note that if we can give the factorization for the independence
polynomial, then the next result will be very useful in solving
unimodality problems for independence polynomials. We refer the
reader to Stanley's survey article~\cite{Sta89} for further
information.
\begin{lem}\label{product}
Let $f(x)$ and $g(x)$ be polynomials with positive coefficients.
\begin{itemize}
\item [\rm (i)] If both $f(x)$ and $g(x)$ have only real zeros, then
so does their product $f(x)g(x)$.
\item [\rm (ii)] If both $f(x)$ and $g(x)$ are log-concave, then
so is their product $f(x)g(x)$.
\item [\rm (iii)] If $f(x)$ is log-concave and $g(x)$ is unimodal, then
their product $f(x)g(x)$ is unimodal.
\item [\rm (iv)] If both $f(x)$ and $g(x)$ are symmetric and unimodal, then
so is their product $f(x)g(x)$. In addition, the mode of $f(x)g(x)$
is the sum of modes of $f(x)$ and $g(x)$.
\end{itemize}
\end{lem}

Let $P(x)$ be a real polynomial of degree $n$. Define its {\it
reciprocal polynomial} by
$$P^*(x):=x^nP\left(\frac{1}{x}\right).$$ The following facts are
elementary but very useful in the sequel.
\begin{itemize}
\item[\rm (i)]
If $P(0)\neq 0$, then $\deg~P^*(x)=\deg~P(x)$ and $(P^*)^*(x)=P(x)$.
\item[\rm (ii)]
$P(x)$ has only real zeros if and only if $P^*(x)$ has only real
zeros.
\item[\rm (iii)]
$P(x)$ is log-concave if and only if $P^*(x)$ is log-concave.
\item[\rm (iv)]
$P(x)$ is symmetric if and only if $P^*(x)=P(x)$.
\end{itemize}
\subsection{Symmetry and unimodality of independence polynomials}
The next result gives a characterization of the graphs having the
symmetric or unimodal independence polynomials.
\begin{prop}\label{prop-semmetry-CCP}
Given two graphs $G$ and $H$, let $\mathscr{C}$ be a clique cover of
$G$ and $U$ be a subset of $V(H)$. Let $|\mathscr{C}|=q$ and
$\alpha(H)=\alpha(H-U)+2$. Then we have the following.
\begin{itemize}
\item[\rm (i)]
If both $I(H;x)$ and $I(H-U;x)$ are symmetric, then so is
$I(G^{\mathscr{C}}\star H^U;x)$.
\item[\rm (ii)]
If both $I(H;x)$ and $I(H-U;x)$ are symmetric and unimodal, then so
is $I(G^{\mathscr{C}}\star H^U;x)$.
\end{itemize}
\end{prop}
\begin{proof}
To show that $I(G^{\mathscr{C}}\star H^U;x)$ is symmetric, we only
need to prove that $I^*(G^{\mathscr{C}}\star
H^U;x)=I(G^{\mathscr{C}}\star H^U;x)$. In view of
Theorem~\ref{thm-CCP}, we have
$$I(G^{\mathscr{C}}\star H^U;x)=I^q(H;x)I\left(G;\frac{x I(H-U;x)}{I(H;x)}\right).$$ Thus it
is clear that
\begin{eqnarray*}
I^*(G^{\mathscr{C}}\star H^U;x)&=&x^{q\alpha(H)}I\left(G^{\mathscr{C}}\star H^U;\frac{1}{x}\right)\\
&=&[I^*(H;x)]^qI\left(G;\frac{\frac{1}{x}I(H-U;\frac{1}{x})}{I(H;\frac{1}{x})}\right)\\
&=&[I^*(H;x)]^qI\left(G;\frac{x I^*(H-U;x)}{I^*(H;x)}\right)\\
&=&I^q(H;x)I\left(G;\frac{x I(H-U;x)}{I(H;x)}\right)\\
&=&I\left(G^{\mathscr{C}}\star H^U;x\right)
\end{eqnarray*}
since $I^*(H;x)=I(H;x)$ and $I^*(H-U;x)=I(H-U;x)$. Thus, we have
shown that (i) holds.

Now we will prove that (ii) holds. Assume that
$$I(G;x)=\sum_{i=0}^{\alpha(G)} s_i x^i.$$ In view of
Theorem~\ref{thm-CCP}, we have
\begin{eqnarray}\label{symm}
I(G^{\mathscr{C}}\star
H^U;x)=I^{q-\alpha(G)}(H;x)\sum_{i=0}^{\alpha(G)} s_i \left[x
I(H-U;x)\right]^i[I(H;x)]^{\alpha(G)-i}.
\end{eqnarray}
 By Lemma~\ref{product} (iv), we can deduce that $\left[
I(H-U;x)\right]^i[I(H;x)]^{\alpha(G)-i}$ is symmetric and unimodal
and the mode of $\left[x I(H-U;x)\right]^i[I(H;x)]^{\alpha(G)-i}$ is
equal to $\frac{\alpha(G)\alpha(H)}{2}$ for any $i\in[0,\alpha(G)]$.
Thus, we obtain that
\begin{eqnarray*}
\sum_{i=0}^{\alpha(G)} s_i \left[x
I(H-U;x)\right]^i[I(H;x)]^{\alpha(G)-i}
\end{eqnarray*} is symmetric and unimodal. Hence, $I(G^{\mathscr{C}}\star H^U;x)$ is
symmetric and unimodal by virtue of equality (\ref{symm}) and
Lemma~\ref{product} (iv).
\end{proof}
Let $k\geq 1$ and $d\geq 0$. Following Bahls and
Salazar~\cite{BS10}, the $K_t$-path of length $k$, denoted by $P(t,
k)$, is the graph $(V,E)$ in which $V = \{v_1, v_2,\ldots,
v_{t+k-1}\}$ and
$$ E = \left\{\{v_i, v_{i+j}\} |1 \leq i \leq t + k - 2, 1 \leq j
\leq \min\{t -1, t + k - i-1\}\right\}.$$ Such a graph consists of
$k$ copies of $K_t$, each glued to the previous one by identifying
certain prescribed subgraphs isomorphic to $K_{t-1}$. The
d-augmented $K_t$ path, denoted  $P(t, k, d)$, is obtained from
$P(t, k)$ by adding new vertices $\left\{u_{i,1}, u_{i,2},\ldots,
u_{i,d}\right\}_{i=0}^{t+k-2}$ and edges
$$\left\{\{v_iu_{i,j}\}, \{v_{i+1} u_{i,j}\} | j =
1,\ldots,d\right\}_{i=1}^{t+k-2} \bigcup \left\{\{v_1u_{0,j}\} | j =
1,\ldots, d\right\}.$$ By the complicated computation, Bahls and
Salazar~\cite{BS10} showed the following corollary, which clearly
follows from Proposition~\ref{prop-semmetry-CCP} (ii). In fact, we
only need to assume that $I(G-v;x)$ and $I(G-N[v];x)$ are symmetric
and unimodal, and that $\alpha(G)= \alpha(G-N[v]) + 2$.
\begin{coro}\emph{\cite[Theorem 2.1]{BS10}}\label{coro+Bahls and Salazar}
Given a graph $G$ and the vertex $v \in V(G)$,  assume that
$I(G-v;x)$, $I(G;x)$ and $I(G-N[v];x)$ are symmetric and unimodal,
and that $\alpha(G-v)=\alpha(G) = \alpha(G-N[v]) + 2$. Then
$I(P(t,k,d)^{V(P(t,k,d))} \star (G-v)^{N(v)}; x)$ is symmetric and
unimodal for any $t \geq2$, $k\geq 1$ and $d\geq0$.
\end{coro}
Actually, Proposition~\ref{prop-semmetry-CCP} (ii) has given a
generalized answer to the following question of Bahls and
Salazar~\cite{BS10}.
\begin{prob}\emph{\cite[Question 5.2]{BS10}}
Can we obtain an explicit description of the most general family of
graphs to which Corollary~\ref{coro+Bahls and Salazar} may be
applied in order to show symmetry and unimodality $?$
\end{prob}
Note that if a graph $G$ satisfies the conditions of
Theorem~\ref{ste}, then we have
$$I(G;x)=\sum_{k=0}^{\lfloor|S|/2\rfloor}i_k(G[V-S])x^k(x+1)^{|S|-2k}$$
using the similar method of Theorem~\ref{thm-CCP}. From the proof of
Proposition~\ref{prop-semmetry-CCP} (ii), it can be seen that the
following results hold, which strengthens Stevanovi\'c's
results~\cite{Ste98}.
\begin{thm}
If there is an independent set $S$ in $G$ such that $|N(A)\cap
S|=2|A|$ holds for every independent set $A \subseteq V (G)-S$, then
$I(G; x)$ is symmetric and unimodal.
\end{thm}

\begin{coro}
\begin{itemize}
\item [\rm (i)]
If a graph $G$ has $i_{\alpha(G)}(G)=1$, $i_{\alpha(G)-1}(G)=|V(G)|$
and the unique maximum independence set $S$ satisfies $|N(u)\cap
S|=2$ for every vertex $u\in V-S$, then $I(G;x)$ is symmetric and
unimodal.
\item [\rm (ii)]
If $G$ is a claw-free graph with $i_{\alpha(G)}=1$, $i_{\alpha(G)-1}
= |V (G)|$, then $I(G; x)$ is symmetric and unimodal.
\end{itemize}
\end{coro}

\subsection{Log-concavity and reality of zeros of independence polynomials}
The following criterion for log-concavity is very useful.
\begin{lem}\emph{\cite[Brenti]{Bre94}}\label{lem-br}
If $P(x)$ is a log-concave polynomial with nonnegative coefficients
and with no internal zeros, then $P(x+r)$ is log-concave for any
positive integer $r$.
\end{lem}

 Similar to Proposition~\ref{prop-semmetry-CCP}, we can
also demonstrate the following result.
\begin{prop}\label{prop-real-CCP}
Given two graphs $G$ and $H$, let $\mathscr{C}$ be a clique cover of
$G$ and $U\subseteq V(H)$. Let $I(H;x)=I(H-U;x)(ax^2+bx+1)$, where
$a,b$ are nonnegative integer.
\begin{itemize}
\item[\rm (i)]
If both $I(G;x)$ and $I(H;x)$ have only real zeros, then so does
$I(G^{\mathscr{C}}\star H^U;x)$.
\item[\rm (ii)]
Assume that $I(G;x)$ has only real zeros. If both $I(H-U;x)$ and
$ax^2+bx+1$ are log-concave, then so is $I(G^{\mathscr{C}}\star
H^U;x)$.
\item[\rm (iii)]
Assume that $I(G;x)$ is log-concave and  $a=0$. If $I(H-U;x)$ is
log-concave, then so is $I(G^{\mathscr{C}}\star H^U;x)$.
\end{itemize}
\end{prop}
\begin{proof}
We first show that (i) and (ii) hold. Let $|\mathscr{C}|=q$ and
$I(G;x)=\prod_{i=1}^{\alpha(G)}(a_ix+1)$ since $I(G;x)$ has only
real zeros. By Theorem~\ref{thm-CCP}, we have
\begin{eqnarray*}
I(G^{\mathscr{C}}\star H^U;x)&=&I^q(H;x)\prod_{i=1}^{\alpha(G)}\left(1+\frac{a_ix}{ax^2+bx+1}\right)\\
&=&I^{q-\alpha(G)}(H;x)I^{\alpha(G)}(H-U;x)\prod_{i=1}^{\alpha(G)}\left[ax^2+(b+a_i)x+1\right].
\end{eqnarray*}
Thus, we obtain that (i) and (ii) follow from Lemma~\ref{lem+comp}
(i) and (ii), respectively.

In what follows, we will prove that (iii) holds. To show
log-concavity of $I(G^{\mathscr{C}}\star H^U;x)$, it suffices to
show that $I^*(G^{\mathscr{C}}\star H^U;x)$ is log-concave. Since
$ax^2+bx+1=1+bx$ for $a=0$, we have
$$I(G^{\mathscr{C}}\star H^U;x)=I^q(H;x) I\left(G;\frac{x}{1+bx}\right)$$ by virtue of
Theorem~\ref{thm-CCP}. Thus we obtain that
 \begin{eqnarray}\label{I*}
I^*(G^{\mathscr{C}}\star H^U;x)
&=&x^{q\alpha(H)}I\left(G^{\mathscr{C}}\star H^U;\frac{1}{x}\right)\nonumber\\
&=&\left[I^*(H;x)\right]^qI\left(G;\frac{1}{x+b}\right)\nonumber\\
&=&\left[I^*(H-U;x)(x+b)\right]^qI\left(G;\frac{1}{x+b}\right)\nonumber\\
&=&\left[I^*(H-U;x)\right]^q(x+b)^{q-\alpha(G)}I^*(G;x+b).
\end{eqnarray}
Applying Lemma~\ref{lem-br} and Lemma~\ref{product} (ii) to the
equality (\ref{I*}), we obtain that log-concavity of
$I^*(G^{\mathscr{C}}\star H^U;x)$ follows from log-concavity of
$I^*(H-U;x)$ and $I^*(G;x)$. Therefore the proof is complete.
\end{proof}

Now we give a simple application of Proposition~\ref{prop-real-CCP}.

A {\it well-covered spider} graph $S_n$ is one of
$\{K_1,K_2,K_{1,m}\circ K_1: m\geq 1\}$. Levit and
Mandrescu~\cite{LM04CJM} proved that the independence polynomial of
the well-covered spider $S_{n}$ is log-concave. By the complicated
calculation, Chen and Wang~\cite{CW10} demonstrated that
$I(K_{2,n}\circ K_1;x)$ is unimodal and log-concave. Further, they
proposed the following general conjecture, which can be confirmed by
our results.
\begin{conj}[\cite{CW10}]\label{conj+spider}
$I(K_{t,n}\circ K_1;x)$ is log-concave for every $t$ and is
therefore unimodal.
\end{conj}
\begin{proof}
To show that $I(K_{t,n}\circ K_1;x)$ is log-concave, we only need to
prove that $I(K_{t,n};x)$ is log-concave by virtue of
Proposition~\ref{prop-real-CCP} (ii) for $H=K_1$. It is easy to
obtain that
$$I(K_{t,n};x)=(x+1)^t+(1+x)^n-1.$$
Without loss of generality, we can assume $t\leq n$. Thus,
$$(x+1)^t+(1+x)^n=(x+1)^t[1+(1+x)^{n-t}].$$ It follows from Lemma~\ref{product} (ii)
that $(x+1)^t+(1+x)^n$ is log-cancave. Consequently, it is clear
that $I(K_{t,n};x)$ is log-concave. This completes the proof.
\end{proof}
\begin{rem}
The graph $H$ in Propositions~\ref{prop-semmetry-CCP} and
\ref{prop-real-CCP} can be a disconnected graph. In addition, we can
also easily find more examples for $H$ applied to
Propositions~\ref{prop-semmetry-CCP} and \ref{prop-real-CCP}.
\end{rem}
\subsection{Claw-free graphs}
Recently, Chudnovsky and Seymour~\cite{CS07} showed that the
independence polynomial of a claw-free graph has only real zeros.
Recall that two real polynomials $f(x)$ and $g(x)$ are {\it
compatible} if $af(x)+bg(x)$ has only real zeros for all $a,b\ge 0$.
Chudnovsky and Seymour actually proved the next result.
\begin{lem}[\cite{CS07}]\label{lem+comp}
Let $G$ be a claw-free graph. Then
\begin{itemize}
\item [\rm (i)]
$I(G-v;x)$ and $xI(G-N[v];x)$ are compatible for any $v\in V(G)$;
\item [\rm (ii)]
$I(G;x)$ has only real zeros.
\end{itemize}
\end{lem}

In \cite{B10}, Bahls proved the following result. In fat, it clearly
follows from Proposition~\ref{prop-real-CCP} (i) and
Lemma~\ref{lem+comp} since $P(t, k)$ is a claw-free graph.
\begin{coro}
Given a graph $G$ and $U \subset V(G)$, assume that $I(G;x)$ has
only real zeros and $I(G;x)=I(G-U;x)(ax^2+bx+1)$, where
$0<a,b\in\mathbb{N}$. Then $I(P(t, n)^{V(P(t, n))} \star G^U; x)$
has only real zeros and is therefore log-concave for any $t \geq2$
and $n\geq1$.
\end{coro}

Noting that the graph $H$ is claw-free for $\alpha(H)\leq 2$, we
immediately obtain the following corollary by virtue of
Proposition~\ref{prop-semmetry-CCP} and
Proposition~\ref{prop-real-CCP}.
\begin{coro}\label{coro-coron}
Let $H$ be a graph with $\alpha(H)\leq 2$ and ${\mathscr{C}}$ be a
clique cover of another graph $G$. Let $P_3$ denote the path with
three vertices and $p\geq 1$.
\begin{itemize}
\item[\rm (i)]
 If $I(G;x)$ has only real zeros, then so
does $I(G^{\mathscr{C}}\star H^{V(H)};x)$. In particular, if
$I(G;x)$ has only real zeros, then so do $I(G\circ H;x)$ and $I(G
\overline{\circ} K_p;x)$.
\item[\rm (ii)]
If $I(G;x)$ is log-concave, then so is $I(G^{\mathscr{C}}\star
K_p^{V(K_p)};x)$. In particular, if $I(G;x)$ is log-concave, then
$I(G\circ K_p;x)$ is log-concave.
\item[\rm (iii)]
If $H=K_p-e$ and $U=V(K_p-e)$, then $I(G^{\mathscr{C}}\star H^U;x)$
is symmetric and unimodal for $p\geq 2$. In particular,
$I(G\circ{2K_1};x)$ and $I(G\circ{P_3};x)$ are symmetric and
unimodal.
\end{itemize}
\end{coro}
\begin{rem}
If $H=K_2-e=2K_1$ and $U=V(2K_1)$, then $I(G^{\mathscr{C}}\star
H^U;x)$ is symmetric, which was also obtained by
Stevanovi\'c~\cite{Ste98}. Mandrescu~\cite{Man09} demonstrated a
particular example of Corollary~\ref{coro-coron} (i): If $I(G;x)$
has only real zeros, then so does $I(G\circ (K_p\bigcup K_q);x)$.
Rosenfeld~\cite{R09} also showed that if $I(G;x)$ has only real
zeros, then so does $I(G \overline{\circ} K_p;x)$.
\end{rem}


\begin{exm}{Centipede Graphs and Caterpillar Graphs}

Let $W_n$ (see Figure 1) and $H_n$ (see Figure 2) be the centipede
graph and the caterpillar graph, respectively. Levit and
Mandrescu~\cite{LM02} showed that $I(W_{n};x)$ is unimodal and
further conjectured that $I(W_{n};x)$ has only real zeros.
Zhu~\cite{Zhu07} settled the conjecture and demonstrated that
$I(H_{n};x)$ is symmetric and unimodal. Now we can use
Corollary~\ref{coro-coron} to give a unified proof of these results.
Since $P_n$, {\it i.e.}, the path with $n$ vertices, is a claw-free
graph. We obtain that $I(W_n;x)=I(P_n\circ K_1;x)$ has only real
zeros, and $I(H_n;x)=I(P_n\circ 2K_1;x)$ has only real zeros and is
symmetric by Corollary~\ref{coro-coron}. Thus $I(H_n;x)$ is
log-concave and unimodal. In fact, we can further obtain that the
independence polynomials of more graphs have only real zeros using
Corollary~\ref{coro-coron} repeatedly.
\end{exm}
\begin{center}
\setlength{\unitlength}{1cm}
\begin{picture}(20,2.5)(-1,-2.0)
\thicklines \put(1,0){\line(1,0){2.15}}
 \thicklines\put(3.85,0){\line(1,0){1.15}}
\put(1,0){\circle*{0.2}}\put(2,0){\circle*{0.2}}\put(3,0){\circle*{0.2}}
\put(4,0){\circle*{0.2}}\put(5,0){\circle*{0.2}}
\put(3.3,0){\circle*{0.1}}\put(3.5,0){\circle*{0.1}}\put(3.7,0){\circle*{0.1}}
 \put(1.0,0.2){$v_1$} \put(2.0,0.2){$v_2$} \put(3.0,0.2){$v_3$}
\put(4.0,0.2){$v_{n-1}$} \put(5.0,0.2){$v_n$}
 \put(1,0){\line(0,-1){1}} \put(2,0){\line(0,-1){1}}
 \put(3,0){\line(0,-1){1}} \put(4,0){\line(0,-1){1}}\put(5,0){\line(0,-1){1}}
\put(1,-1){\circle*{0.2}}\put(2,-1){\circle*{0.2}}
\put(3,-1){\circle*{0.2}}\put(4,-1){\circle*{0.2}}
\put(5,-1){\circle*{0.2}}
\put(2.5,-2.0){Figure~1: $W_n$}
\thicklines \put(7,0){\line(1,0){4.4}}
 \thicklines\put(12.6,0){\line(1,0){0.4}}
\put(7,0){\circle*{0.2}}\put(9,0){\circle*{0.2}}
\put(11,0){\circle*{0.2}}\put(13.0,0){\circle*{0.2}}
\put(11.7,0){\circle*{0.1}}\put(12,0){\circle*{0.1}}
\put(12.3,0){\circle*{0.1}}\put(12.3,0){\circle*{0.1}}
 \put(7.0,0.2){$v_1$} \put(9,0.2){$v_2$}\put(11,0.2){$v_3$}
\put(13.0,0.2){$v_{n}$}
 \put(7,0){\line(1,-1){0.7}}\put(7,0){\line(-1,-1){0.7}}
 \put(9,0){\line(1,-1){0.7}} \put(9,0){\line(-1,-1){0.7}}
 \put(11,0){\line(-1,-1){0.7}}\put(11,0){\line(1,-1){0.7}}
 \put(13,0){\line(1,-1){0.7}}\put(13,0){\line(-1,-1){0.7}}

\put(6.3,-0.7){\circle*{0.2}}\put(7.7,-0.7){\circle*{0.2}}
\put(8.3,-0.7){\circle*{0.2}}\put(9.7,-0.7){\circle*{0.2}}
\put(10.3,-0.7){\circle*{0.2}}\put(11.7,-0.7){\circle*{0.2}}
\put(12.3,-0.7){\circle*{0.2}}\put(13.7,-0.7){\circle*{0.2}}
\put(9.5,-2.0){Figure~2: $H_n$}
\end{picture}
\end{center}

\begin{exm}{$N$-sunlet Graphs}

The {\it $n$-sunlet graph} is the graph with $2n$ vertices obtained
by attaching pendant edges to a cycle graph, {\it i.e.}, $C_n\circ
K_1$, where $C_n$ is the cycle with $n$ vertices. Applying
Corollary~\ref{coro-coron}, we have $I(C_n\circ K_1;x)$ has only
real zeros since $C_n$ is a claw-free graph. Therefore $I(C_n\circ
K_1;x)$ is log-concave and unimodal. In addition, we also can verify
that $I(C_n\circ K_p;x)$ and $I(C_n\circ 2K_p;x)$ have only real
zeros for $p\geq 1$.
\end{exm}

\begin{exm}{A Conjecture of Levit and Mandrescu} \hspace*{\parindent}

In \cite{LM07}, Levit and Mandrescu constructed a family of graphs
$H_n$ from the path $P_n$ by the ``clique cover construction'', as
shown in Figure 3. By $H_0$ we mean the null graph.
\begin{center}
\setlength{\unitlength}{1cm}
\begin{picture}(20,2.5)(-1.5,-1.5)

\thicklines \put(0,0){\line(1,0){3.15}} \thicklines
\put(3.8,0){\line(1,0){2.2}}
\put(0,0){\circle*{0.2}}\put(1,0){\circle*{0.2}}
\put(2,0){\circle*{0.2}}\put(3,0){\circle*{0.2}}
\put(3.3,0){\circle*{0.1}}\put(3.5,0){\circle*{0.1}}
\put(3.7,0){\circle*{0.1}}
 \put(4,0){\circle*{0.2}}
\put(5,0){\circle*{0.2}}\put(6,0){\circle*{0.2}}

 \put(0,0){\line(0,1){1}} \put(0,0){\line(0,-1){1}}
 \put(1,0){\line(-1,-1){1}} \put(1,0){\line(-1,1){1}}
 \put(2,0){\line(0,-1){1}}\put(2,0){\line(0,1){1}}
 \put(3,0){\line(-1,-1){1}}\put(3,0){\line(-1,1){1}}
 \put(4,0){\line(0,-1){1}}\put(4,0){\line(0,1){1}}
 \put(5,0){\line(-1,-1){1}}\put(5,0){\line(-1,1){1}}
 \put(6,0){\line(0,-1){1}}\put(6,0){\line(0,1){1}}
 \put(0,-1){\circle*{0.2}}\put(0,1){\circle*{0.2}}
 \put(2,-1){\circle*{0.2}}\put(2,1){\circle*{0.2}}
\put(4,-1){\circle*{0.2}}\put(4,1){\circle*{0.2}}
\put(6,-1){\circle*{0.2}}\put(6,1){\circle*{0.2}}

\put(0.1,0.1){$v_{1}$}\put(1.5,-0.3){$e$}
\put(1.1,0.1){$v_{2}$}\put(2.1,0.1){$v_{3}$}
\put(5.1,0.1){$v_{2m}$}\put(6.1,0.1){$v_{2m+1}$}
\put(2,-2){$H_{2m+1}$}

\thicklines \put(8,0){\line(1,0){3.15}} \thicklines
\put(11.8,0){\line(1,0){1.2}}
\put(8,0){\circle*{0.2}} \put(9,0){\circle*{0.2}}
\put(10,0){\circle*{0.2}}\put(11,0){\circle*{0.2}}
\put(11.3,0){\circle*{0.1}}\put(11.5,0){\circle*{0.1}}
\put(11.7,0){\circle*{0.1}}
\put(12,0){\circle*{0.2}}\put(13,0){\circle*{0.2}}

\put(8,0){\line(1,-1){1}}\put(8,0){\line(1,1){1}}
 \put(9,0){\line(0,-1){1}}\put(9,0){\line(0,1){1}}
 \put(10,0){\line(1,-1){1}}\put(10,0){\line(1,1){1}}
 \put(11,0){\line(0,-1){1}}\put(11,0){\line(0,1){1}}
 \put(12,0){\line(1,-1){1}}\put(12,0){\line(1,1){1}}
 \put(13,0){\line(0,-1){1}}\put(13,0){\line(0,1){1}}

 \put(9,-1){\circle*{0.2}}\put(9,1){\circle*{0.2}}
\put(11,-1){\circle*{0.2}}\put(11,1){\circle*{0.2}}
\put(13,-1){\circle*{0.2}}\put(13,1){\circle*{0.2}}

\put(8.28,0.1){$v_1$}\put(9.1,0.1){$v_2$} \put(9.5,-0.3){$e$}
\put(10.28,0.1){$v_3$}\put(11.1,0.15){$v_4$}
\put(13.1,0.1){$v_{2m}$} \put(10,-2){$H_{2m}$}
\end{picture}\\[5mm]
Figure~3
\end{center}

Levit and Mandrescu~\cite{LM07} proved that the independence
polynomials of $H_n$ is symmetric and unimodal. They further made
the next conjecture.
\begin{conj}[{\cite{LM07}}]
\label{conj-LM} The independence polynomial of $H_n$ is log-concave
and has only real zeros for $n\ge 1$.
\end{conj}
Recently, this conjecture has been confirmed by Wang and
Zhu~\cite{WZ10}.
Now we also can use Corollary~\ref{coro-coron} to give a simple
proof. In view of Corollary~\ref{coro-coron}, we easily see that
$I(H_n;x)$ has only real zeros since $P_n$ is claw-free and
$\alpha(H)=\alpha(2K_1)=2$. Consequently, $I(H_n;x)$ is log-concave
and unimodal.
\end{exm}

Finally, we also show the following general result, which preservers
the reality of zeros of the independence polynomial. In particular,
it extends the result of Chudnovsky and Seymour~\cite{CS07}.
\begin{prop}\label{prop-root product}
Let $\mathscr{C}$ be a clique cover of a graph $G$ and
$|\mathscr{C}|=q$. Assume that $H$ is a claw-free graph with the
root $v$ and $U=N(v)$. If $I(G;x)$ has only real zeros, then so does
$I(G^{\mathscr{C}}\star (H-v)^U;x)$. In particular, if $I(G;x)$ has
only real zeros, then so does $I(G \overline{\circ} H;x)$.
\end{prop}
\begin{proof}
 Let $I(G;x)=\prod_{i=1}^{\alpha(G)}(r_ix+1)$ since $I(G;x)$ has
only real zeros. Thus by Theorem~\ref{thm-CCP} we obtain that
\begin{eqnarray*}
I(G^{\mathscr{C}}\star (H-v)^U;x)
&=&I^{q-\alpha(G)}(H-v;x)\prod_{i=1}^{\alpha(G)}\left[I(H-v;x)+r_ixI(H-N[v];x)\right].
\end{eqnarray*}
According to the hypothesis, the graph $H$ is claw-free, so are the
induced graphs $H-v$ and $H-N[v]$. By virtue of
Lemma~\ref{lem+comp}, $I(H-v;x)$ and $xI(H-N[v];x)$ have only real
zeros and are compatible. Hence $I(H-v;x)+r_ixI(H-N[v];x)$ has only
real zeros for every $i$. Consequently, $I(G^{\mathscr{C}}\star
(H-v)^U;x)$ has only real zeros, which is desired.
\end{proof}
In particular, as a corollary of Proposition~\ref{prop-root
product}, we obtain the next result.
\begin{coro}\emph{\cite[Proposition 3.3]{WZ10}}
Let $G$ be a rooted claw-free graph and $P_n$ be the path with $n$
vertices. Then $I(P_n\overline{\circ} G;x)$ has only real zeros.
\end{coro}
\begin{rem}
Proposition~\ref{prop-root product} can be repeatedly used to
construct infinite graphs with claw, whose independence polynomials
have only real zeros.
\end{rem}


\section{Concluding Remarks}
\hspace*{\parindent}
In this paper we define the clique cover product of graphs. Based on
the formula of the independence polynomial of the graph so-formed,
we show some operations of graphs, which preserve symmetry,
unimodality, log-concavity or reality of zeros of the independence
polynomials.

In fact, we can find that $I(T;x)$ has only real zeros or is
log-concave if we calculate the independence polynomial of the tree
$T$ with fewer vertices. Therefore we can confirm the unimodality
conjectures in the literature for more graphs using the construction
$T\circ K_1$ and $T\circ 2K_1$ repeatedly. In particular, $I(T\circ
2K_1;x)$ is symmetric and unimodal for any tree $T$. Generally
speaking, if we can give the factorization for
$I(G^{\mathscr{C}}\star H^U;x)$ in Theorem~\ref{thm-CCP} and can
show that its every factor is symmetry or log-concave, or has only
real zeros, then we can obtain that $I(G^{\mathscr{C}}\star H^U;x)$
has the same property in view of Lemma~\ref{product}. In particular,
if both $G$ and $H$ are some particular graphs, then such result may
easily be proved.

From this paper, we can see that it is a good way to construct
graphs with independence polynomial being symmetric, unimodal or
log-concave, or having only real zeros by using an operation of
graphs. Thus, it is possible to define some other operations of
graphs.  Recall that a {\it cycle cover} of a graph $G$ is a
spanning graph of $G$, each connected component of which is a vertex
(called a {\it vertex-cycle}), an edge (called an {\it edge-cycle}),
or a {\it proper cycle}. Here we give a new operation of graphs
called the {\it cycle cover product}: Given two graphs $G$ and $H$,
let $\Gamma$ be a cycle cover of $G$ and $U\subseteq V(H)$.
Construct a new graph, denoted by $G^\Gamma\otimes H^U$, as follows:
If a cycle $C\in \Gamma$ is
\begin{itemize}
\item[\rm (i)]
a vertex-cycle, say $v$, then add two copies of $H$ and join each
vertex in two $U$ to $v$;
\item[\rm (ii)]
an edge-cycle, say $uv$, then add two copies of $H$ and join each
vertex in two $U$ to both $u$ and $v$;
\item[\rm (iii)] a proper cycle, with $V(C)= \{v_i: 1\leq i \leq s\}$,
$E(C)=\{v_iv_{i+1}:1\leq i \leq s-1\}\cup\{v_1v_s\}$, then add $s$
copies of $H$, say $\{H_i: 1 \leq i \leq s\}$ and each vertex in
$i$th copy of $U$ is joined to two consecutive vertices
$v_i,v_{i+1}$ on $C$ (each vertex in $s$th copy of $U$ is joined to
two consecutive vertices $v_{s},v_1$ on $C$).
\end{itemize}
Actually, using the similar technique of the clique cover product of
graphs, we can also prove the following result. We here omit its
proof for brevity.

\begin{thm}\label{thm-gsmm}
Given two graphs $G$ and $H$, assume that $\Gamma$ is a cycle cover
of $G$ containing $k$ vertex-cycles and $U$ is a subset of $V(H)$.
Let $|V(G)|=n$.
\begin{itemize}
\item[\rm (i)]
The independence polynomial $I(G^\Gamma\otimes
H^U;x)=I^{n+k}(H;x)I\left(G;\frac{xI^2(H-U;x)}{I^2(H;x)}\right)$.
\item[\rm (ii)]
Let $\alpha(H)=\alpha(H-U)+1$. If both $I(H;x)$ and $I(H-U;x)$ are
symmetric, then so is $I(G^\Gamma\otimes H^U;x)$.
\item[\rm (iii)]
Let $\alpha(H)=\alpha(H-U)+1$. If both $I(H;x)$ and $I(H-U;x)$ are
symmetric and unimodal, then so is $I(G^\Gamma\otimes H^U;x)$.
\item[\rm (iv)]
If $I(G;x)$ has only real zeros and $I(H;x)=I(H-U;x)(1+ax)$, then
$I(G^\Gamma\otimes H^U;x)$ has only real zeros.
\item[\rm (v)]
If $I(G;x)$ has only real zeros and $H=K_p$, then $I(G^\Gamma\otimes
H^{V(K_p)};x)$ has only real zeros.
\end{itemize}
\end{thm}
\begin{rem}
From (ii) of this theorem, it is clear that $I(G^\Gamma\otimes
K_1^{V(K_1)};x)$ is symmetric, which also obtained by
Stevanovi\'c~\cite{Ste98} using the ``cycle cover construction". In
fact, $I(G^\Gamma\otimes K_1^{V(K_1)};x)$ is symmetric and unimodal
by (iii) of this theorem.
\end{rem}

\begin{rem}
Given two graphs $G$ and $H$, assume that $\Gamma$ is a cycle cover
of $G$ containing no proper cycle, $\mathscr{C}$ is a clique cover
of $G$ only containing $K_1,K_2$ and $U$ is a subset of $V(H)$. Then
it is obvious that
$$I(G^{\mathscr{C}}\star (2H)^{2U};x)=I(G^{\Gamma}\otimes H^U;x).$$
\end{rem}
At the end of this paper, we refer the reader to
\cite{Bra06,Bre94,Bre89,LW07,MW08,Sta89,WYjcta05,WYeujc05,WYjcta07}
for more results about unimodality problems of sequences and
polynomials.


\end{document}